\documentclass[twoside,11pt]{amsart}
  \usepackage{mabliautoref}
  \usepackage{amsmath,amsfonts,amssymb,amsthm,mathtools,multirow,tabls,mathrsfs}
  \usepackage[a4paper,margin=1.1in]{geometry}  
  \usepackage{lmodern}
  \usepackage[all]{xy}
  \usepackage[utf8]{inputenc}
  \usepackage[T1]{fontenc}
  \usepackage[shortlabels]{enumitem}
  \usepackage{stmaryrd}

  \numberwithin{equation}{section}

\RequirePackage[dvipsnames,usenames]{xcolor}
\usepackage{tikz}
\usepackage{tikz-cd}
\usetikzlibrary{decorations.markings}
\tikzset{degil/.style={
  decoration={markings,
  mark= at position 0.5 with {
  \node[transform shape] (tempnode) {$\backslash$};
  \draw[thick] (tempnode.north east) -- (tempnode.south west);
  }}, postaction={decorate}
}}
\usetikzlibrary{decorations.text, calc, matrix, arrows, decorations.pathreplacing}
\usepackage{enumitem}
\usepackage{fdsymbol}
\usepackage{minibox}
\usepackage[framemethod=TikZ]{mdframed}
\usepackage{blkarray}
\usepackage{ifthen}

\setlist[itemize,enumerate]{leftmargin=0.9cm}

\mdfsetup{
 skipabove=6pt,
 skipbelow=3pt,
 roundcorner=10pt
}

\hypersetup{
bookmarks,
bookmarksdepth=3,
bookmarksopen,
bookmarksnumbered,
pdfstartview=FitH,
colorlinks,backref,hyperindex,
linkcolor=Sepia,
anchorcolor=BurntOrange,
citecolor=MidnightBlue,
citecolor=OliveGreen,
filecolor=BlueViolet,
menucolor=Yellow,
urlcolor=OliveGreen
}

\DeclareMathAlphabet{\mathchanc}{OT1}{pzc}%
                                 {m}{it}

\newcommand{\bA}{\mathbb{A}}

\newcommand{\bE}{\mathbb{E}}

\newcommand{\bP}{\mathbb{P}}
\newcommand{\bQ}{\mathbb{Q}}
\newcommand{\bR}{\mathbb{R}}

\newcommand{\bZ}{\mathbb{Z}}

\newcommand{\cC}{\mathcal{C}}

\newcommand{\cI}{\mathcal{I}}

\newcommand{\cK}{\mathcal{K}}
\newcommand{\cL}{\mathcal{L}}

\newcommand{\cO}{\mathcal{O}}

\newcommand{\wt}{\widetilde}
\newcommand{\wht}{\widehat}

\DeclareMathOperator{\Div}{Div}

\DeclareMathOperator{\length}{{length}}

\DeclareMathOperator{\mult}{mult}

\DeclareMathOperator{\ord}{{ord}}

\DeclareMathOperator{\Sing}{{Sing}}

\DeclareMathOperator{\Supp}{{Supp}}

\newcommand{\epsint}{\varepsilon_{\rm int}}
\newcommand{\Ord}[2]{
\ifthenelse{\equal{#2}{no}}
  {
   \operatorname{ord}_{#1}
  }
  {
   \operatorname{ord}_{#1}(#2)
  }
}
\newcommand{\Ard}[2]{
\ifthenelse{\equal{#2}{no}}
  {
   \mathop{\overline{\ord}_{#1}}
  }
  {
   \mathop{\overline{\ord}_{#1}}(#2)
  }
}
\newcommand{\Mult}[3]{
\ifthenelse{\equal{#3}{no}}
  {
   \overline{\operatorname{mult}}_{#1}#2
  }
  {
   \ifthenelse{\equal{#2}{no}}
  {
   \overline{\operatorname{mult}}_{#1}\bullet|_{#3}
  }
  {
   \overline{\operatorname{mult}}_{#1}#2|_{#3}
  }
  }
}

\newcommand{\factor}[2]{\left. \raise 2pt\hbox{\ensuremath{#1}} \right/
        \hskip -2pt\raise -2pt\hbox{\ensuremath{#2}}}

\counterwithin*{equation}{section}
\counterwithin*{equation}{subsection}

\makeatletter
\renewcommand\subsection{
  \renewcommand{\sfdefault}{pag}
  \@startsection{subsection}%
  {2}{0pt}{.8\baselineskip}{.4\baselineskip}{\raggedright
 \sffamily\itshape\small\bfseries
  }}
\renewcommand\section{
  \renewcommand{\sfdefault}{phv}
  \@startsection{section} %
  {1}{0pt}{\baselineskip}{.8\baselineskip}{\centering
    \sffamily
    \scshape
    \bfseries
}}
\makeatother

\renewcommand{\theenumi}{\arabic{enumi}}

\newcommand{\fm}{\mathfrak{m}}

  \usepackage[most]{tcolorbox}
  \usetikzlibrary{shapes.geometric, arrows, positioning, fit, calc,decorations.pathmorphing,shapes}
\tikzstyle{e} = [ellipse, minimum width=1cm, minimum height=0.5cm,text centered, draw=black]
\tikzstyle{arrow} = [thick,-,>=stealth]
\tikzstyle{connection}=[inner sep=0,outer sep=0]

\tikzset{
    vertl/.style={anchor=south, rotate=90, inner sep=.4mm, pos=0.4}
}

\tikzset{
    vertr/.style={anchor=north, rotate=90, inner sep=.7mm, pos=0.4}
}

\tikzset{
    diag/.style={anchor=south, rotate=30, inner sep=.5mm}
}

\tikzset{
    ddiag/.style={anchor=south, rotate=-16, inner sep=0mm}
}

   \author[L. Rösler]{Linus Rösler}
  \address{\'Ecole Polytechnique F\'ed\'erale de Lausanne, Chair of Algebraic Geometry \newline 
    \indent MA C3 615 (Bâtiment MA), Station 8, CH-1015 Lausanne}
  \email{linus.rosler@epfl.ch}

  \title[On Seshadri constants of adjoint divisors on surfaces and threefolds]{On Seshadri constants of adjoint divisors on surfaces and threefolds in arbitrary characteristic}


    \AtBeginDocument{%
   \def\MR#1{}
    }

\begin{document}

\begin{abstract}
    We develop a new approach towards obtaining lower bounds of the Seshadri constants of ample adjoint divisors on smooth projective varieties $X$ in arbitrary characteristic. Let $x\in X$ be a closed point and $A$ an ample divisor on $X$. If $X$ is a surface, we recover some known lower bounds by proving, e.g., that $\varepsilon(K_X+4A;x)\geq 3/4$. If $X$ is a threefold, we prove that for all $\delta>0$ and all but finitely many curves $C$ through $x$, we have $\frac{(K_X+6A).C}{\mult_x C}\geq\frac{1}{2\sqrt{2}}-\delta$. In particular, if $\varepsilon(K_X+6A;x)<1/(2\sqrt{2})$, then $\varepsilon(K_X+6A;x)$ is a rational number, attained by a Seshadri curve $C$.
\end{abstract} 
\maketitle
\tableofcontents
\section{Introduction}\label{sec:intro}
A general phenomenon for notions of positivity of a line bundle on an algebraic variety is that numerical positivity (e.g.,\ ampleness) leads to geometrical positivity (e.g.,\ basepoint--freeness) by passing to a sufficiently high multiple. A natural and broad question is then whether, in certain situations, one can make this effective, i.e.,\ to precisely quantify the meaning of `sufficiently high' in the last sentence. While being interesting already in their own right, results in this area are often key ingredients in establishing boundedness results.

The guiding question of this paper is the following conjecture due to Murayama.
\begin{conjectureintro}[\protect{\cite[Conjecture 1.2.1.(I)]{Murayama_Effective_Fujita-type_theorems_for_surfaces_in_arbitrary_characteristic}}]\label{conj:weakest_Fujita_freness_conjecture}
    There exist functions $e,f\colon\bZ_{>0}\to\bZ_{>0}$ such that for every smooth projective variety $X$ of dimension $n$ over an algebraically closed field and every ample divisor $A$ on $X$, the linear system $|e(n)K_X+f(n)A|$ is basepoint--free.
\end{conjectureintro}

This is a weaker version of Fujita's freeness conjecture (\cite[Conjecture on p. 67]{Fujita_On_polarized_manifolds_whose_adjoint_bundles_are_not_semipositive}), since one is allowed to take arbitrary multiples (depending on the dimension) of both $K_X$ and $A$. In characteristic zero, \autoref{conj:weakest_Fujita_freness_conjecture} is answered affirmatively by a landmark result of Angehrn and Siu in \cite{Angehrn_Siu_Effective_freeness_and_point_separation_for_adjoint_bundles}, namely that taking $e(n)=1$ and $f(n)=\binom{n+1}{2}+1$ works. 

In positive characteristic, only the surface case is well understood. On the one hand, by \cite{Shepherd_Barron_Unstable_vector_bundles_and_linear_systems_on_surfaces_in_characteristic_p, Chen_Fujitas_conjecture_for_quasi_elliptic_surfaces}, Fujita's freeness conjecture holds for smooth surfaces of special type, and on the other hand, Di Cerbo and Fanelli showed in \cite[Theorem 4.12]{Cerbo_Fanelli_Effective_Matsusaka_s_Theorem_for_surfaces_in_characteristic_p} that $|2K_X+tA|$ is basepoint--free for all smooth surfaces $X$ of general type, ample divisors $A$ on $X$ and $t\geq 5$. Combining the two, we see that $|2K_X+6A|$ is basepoint--free in all cases, so in \autoref{conj:weakest_Fujita_freness_conjecture}, the values $e(2)=2$ and $f(2)=6$ work. Note, however, that it is really necessary to have $e(2)>1$: In \cite{Gu_Zhang_Zhang_counterexamples_to_Fujitas_conjecture_on_surfaces_in_positive_characteristic}, by generalizing the construction of Raynaud surfaces, Gu, Zhang, and Zhang obtain for every positive integer $r$ a smooth projective surface $X$ and an ample divisor $A$ on $X$ such that $|K_X+rA|$ has a basepoint. In particular, the pair of values $(e(2),f(2))=(2,6)$ is lexicographically sharp for the case $n=2$ in \autoref{conj:weakest_Fujita_freness_conjecture} (combine the counterexamples in \cite{Gu_Zhang_Zhang_counterexamples_to_Fujitas_conjecture_on_surfaces_in_positive_characteristic} with $X=\bP^2$).

\begin{remintro}
    If \autoref{conj:weakest_Fujita_freness_conjecture} is true, it is natural to ask for the lexicographically minimal value of $(e(n),f(n))$. For $n=1,2$, we see that this is $(n,n(n+1))$, but it is nothing more than a tempting guess that this could hold for arbitrary $n$.
\end{remintro}

In dimension $n\geq 3$ and positive characteristic, essentially nothing is known. Conceptually, the problem is that the approach in \cite{Shepherd_Barron_Unstable_vector_bundles_and_linear_systems_on_surfaces_in_characteristic_p,Cerbo_Fanelli_Effective_Matsusaka_s_Theorem_for_surfaces_in_characteristic_p,Chen_Fujitas_conjecture_for_quasi_elliptic_surfaces} (which is an adaption of Reider's method in \cite{Reider_Vector_bundles_of_rank_2_and_linear_systems_on_algebraic_surfaces}, where he established the surface case in characteristic zero), relies on the Bombieri--Mumford classification of surfaces. In \cite{Murayama_Effective_Fujita-type_theorems_for_surfaces_in_arbitrary_characteristic}, Murayama proposes to prove \autoref{conj:weakest_Fujita_freness_conjecture} by exhibiting a lower bound of Seshadri constants (see \autoref{subsec:Seshadri_constants}).
\begin{conjectureintro}[\protect{\cite[Conjecture 1.3.1]{Murayama_Effective_Fujita-type_theorems_for_surfaces_in_arbitrary_characteristic}}]\label{conj:lower_bound_seshadri_adjoint}
    There exist functions $g,h\colon\bZ_{>0}\to\bZ_{>0}$ and $c\colon\bZ_{>0}\to\bR_{>0}$ such that for every smooth projective variety $X$ of dimension $n$ over an algebraically closed field, every closed point $x\in X$ and every ample divisor $A$ on $X$, we have
    \begin{align*}
        \varepsilon(g(n)K_X+h(n)A;x)\geq c(n).
    \end{align*}
\end{conjectureintro}
Note that \autoref{conj:weakest_Fujita_freness_conjecture} and \autoref{conj:lower_bound_seshadri_adjoint} are basically equivalent (by \cite[Proposition 5.1.19]{Lazarsfeld_Positivity_in_algebraic_geometry_I} resp.\ \cite[Corollary 3.2]{Mustata_Schwede_A_Frobenius_variant_of_Seshadri_constants} and \cite[Example 2.2.3]{Murayama_Seshadri_constants_and_fujitas_conjecture_via_positive_characteristic_methods}). It is known for $n=2$ by taking $g(2)=1$, $h(2)=4$ and $c(2)=1/2$. This was proven in characteristic zero in \cite[Theorem 4.1]{Bauer_Szemberg_on_the_Seshadri_constants_of_adjoint_line_bundles} and then in \cite[Corollary 3.2]{Murayama_Effective_Fujita-type_theorems_for_surfaces_in_arbitrary_characteristic} in arbitrary characteristic, essentially by the same proof. It is important to note that the proof does not rely on the Bombieri--Mumford classification, and thus this provides a classification-- and characteristic--free proof of \autoref{conj:weakest_Fujita_freness_conjecture} for $n=2$. Nevertheless, it is not clear how to generalize their proof to higher dimensions. In characteristic zero, it is observed in \cite[Theorem 3.1]{Bauer_Szemberg_on_the_Seshadri_constants_of_adjoint_line_bundles} that the Seshadri constant of an ample adjoint divisor on a smooth projective variety of dimension $n$ is bounded from below by $c(n)=2/(n^2+n+4)$, so \autoref{conj:lower_bound_seshadri_adjoint} holds if we further take $g(n)=1$ and $h(n)=n+2$. However, their observation relies on the positive answer to Fujita's freeness conjecture given in \cite{Angehrn_Siu_Effective_freeness_and_point_separation_for_adjoint_bundles}, i.e.,\ precisely the type of result we want to achieve through \autoref{conj:lower_bound_seshadri_adjoint}. To summarize, although it seems that \autoref{conj:lower_bound_seshadri_adjoint} could be the way to prove \autoref{conj:weakest_Fujita_freness_conjecture} in higher dimensions, the case of dimension $n\geq 3$ and positive characteristic is wide open also in this regard.

\subsection{Strategy and main results}\label{subsec:strategy_and_main_results}
In this paper, we develop a new approach to \autoref{conj:lower_bound_seshadri_adjoint} in arbitrary dimension. The key point of the strategy is to combine two types of lower bounds for intersection numbers: when the intersection is proper (e.g.,\ \eqref{eq:lb_intersection_intro}), and intersections involving $K_X$ where we can use adjunction (e.g.,\ \eqref{eq:adjunction_intro}). Let us explain this in greater detail.

We start with the following very basic observation (\autoref{lem:intersection_lb_multiplicities}): Let $X$ be a smooth projective variety, $x\in X$ a closed point, and $C$ a curve through $X$. If now $D$ is an effective divisor with $C\not\subseteq\Supp D$, then we have
\begin{equation}\label{eq:lb_intersection_intro}
    D.C\geq\mult_x D\cdot\mult_x C.
\end{equation}
Therefore, if we are given an ample divisor $L$ and want to bound $\varepsilon(L;x)$ from below, we can look for divisors $D$ in $|lL|$ (where $l>0$) with high multiplicity at $x$, and then we get a lower bound
\begin{align*}
    \frac{L.C}{\mult_x C}\geq \frac{1}{l}\mult_x D,
\end{align*}
at least for curves $C\not\subseteq\Supp D$. By considering the $\bQ$--divisor $D'=(1/l)D$ and defining its multiplicity at $x$ as $\mult_x D'\coloneqq(1/l)\mult_x D$ (see \autoref{def:order_Q-divisor} and \autoref{rem:properties_asymptotic_order}), we are naturally led to maximizing $\mult_x$ on $|L|_{\bQ}$ (\autoref{lem:div_high_mult_irred}). But eventually, to get lower bounds for Seshadri constants, we will also have to deal with curves contained in the support of $D$. We will now outline how this can be done in the case of adjoint divisors.

Say for simplicity that $X$ is a surface, and that $L$ is an adjoint divisor of the form $L=K_X+4A$ with $A$ ample. Then as $K_X+3A$ is nef by the Cone Theorem, divisors $D$ in $|A|_{\bQ}$ with large multiplicity at $x$ again lead to lower bounds for Seshadri ratios for curves $C\not\subseteq\Supp D$: we have
\begin{align*}
    \frac{(K_X+4A).C}{\mult_x C}\geq\frac{D.C}{\mult_x C}\geq \mult_x D.
\end{align*}
However, for components of $D$ this does not work, as then the self--intersection $C^2$ appears in $D.C$, and it can be arbitrarily negative. The key idea is now the following: To take care of the components of $D$, we exploit the presence of $K_X$ and use the numerical adjunction formula on surfaces (\cite[Exercise V.1.3]{Hartshorne_Algebraic_geometry}), namely
\begin{align*}
    (K_X+C).C=2p_a(C)-2.
\end{align*}
If we now denote $\mu\coloneqq\mult_x C$, then by \cite[Example V.3.9.2]{Hartshorne_Algebraic_geometry} we obtain
\begin{equation}\label{eq:adjunction_intro}
    (K_X+C).C\geq \mu^2-\mu-2.
\end{equation}
So, if we view $(K_X+C).C$ as a replacement of the self--intersection $C^2$, we obtain a lower bound of order $(\mult_x C)^2$, as if $C$ would intersect itself properly.

The tricky part then is making the term $(K_X+C).C$ appear in $(K_X+3A+D).C$, as $C$ can have an arbitrary coefficient in $D$. We achieve this by finding a suitable convex decomposition of $K_X+3A+D$, and thereby we obtain the following result in \autoref{sec:surface_case}.

\begin{propletter}[\autoref{prop:lower_bound_surfaces}]\label{propA}
    Let $X$ be a smooth projective surface, $x\in X$ a closed point and $A$ an ample divisor on $X$. Then
    \begin{align*}
        \varepsilon(K_X+4A;x)\geq3/4.
    \end{align*}
\end{propletter}

\begin{remintro}
    This bound is not sharp, see \autoref{rem:surface_case}. However, \autoref{prop:lower_bound_surfaces} is actually slightly more general than what is stated above, and this will be useful in the proof of \autoref{thmB}, see \autoref{rem:comp_propA_murayama}
\end{remintro}

The main point of interest of this paper is that the above approach can be generalized to higher dimensions. However, the adjunction argument becomes more complicated, as curves will have higher codimension in this case. But the key idea is the same: for adjoint divisors of the form $K_X+tA$, we take $D\in |A|_{\bQ}$ with large multiplicity at $x$ and replace one $A$ with $D$. For curves $C$ with $C\not\subseteq\Supp D$ we use \eqref{eq:lb_intersection_intro} to obtain a lower bound, and for curves $C\subseteq \Supp D$ we use adjunction (for the component containig $C$). The key technical input here is a replacement of \eqref{eq:adjunction_intro} in higher dimensions (\autoref{lem:bend-and-break}), which we prove using the remarkably general version of Bend--and--Break in \cite[Section II.5]{Kollar_Rational_curves_on_algebraic_varieties}.

Unfortunately, with the current state of the method, we are unable to get a universal lower bound for all curves. However, for threefolds, we obtain an explicit lower bound after excluding finitely many curves, which leads to the following rationality result.

\begin{thmletter}[\autoref{thm:threefold_case}]\label{thmB}
    Let $X$ be a smooth projective threefold over an algebraically closed field of arbitrary characteristic, let $A$ be an ample divisor on $X$ and $x\in X$ a closed point. Let $\delta>0$ be a positive real number. Then for all but finitely many curves $C$ through $x$ we have
    \begin{align*}
        \frac{(K_X+6A).C}{\mult_xC}\geq \frac{1}{2\sqrt{2}}-\delta.
    \end{align*}
    In particular, if $\varepsilon(K_X+6A;x)<1/2\sqrt{2}$, then $\varepsilon(K_X+6A;x)$ is rational and there exists a curve $C$ such that
    \begin{align*}
        \varepsilon(K_X+6A;x)=\frac{(K_X+6A).C}{\mult_xC}.
    \end{align*}
\end{thmletter}
In general, it is known that submaximal Seshadri constants (i.e.,\ $\varepsilon(L;x)<\sqrt[n]{L^n}$) are attained by some strict subvariety $V\subset X$ through $x$, so in particular they are the $d$--th root of a rational number for some $1\leq d<n$ (see \cite[Proposition 4]{Steffens_remarks_on_Seshadri_constants}). Furthermore, for submaximal Seshadri constants on surfaces, there is a subtle improvement by \cite[Proposition 1.1]{Bauer_Szemberg_Seshadri_constants_on_surfaces_of_general_type}, namely that for all $\delta>0$, there are at most finitely many curves $C$ with $L.C/\mult_x C$ below $\sqrt{L^2}-\delta$. Therefore, \autoref{thmB} can be seen as an analog of this result in the threefold case, which applies to ample enough adjoint bundles and where $\sqrt{L^2}$ is replaced by the absolute constant $1/(2\sqrt{2})$.

\subsection{Notation and basic definitions}\label{subsec:notation}

\begin{itemize}
    \item We work over an algebraically closed field $k$ of arbitrary characteristic. 

    \item A variety is an integral projective scheme over $k$. A subvariety is an integral closed subscheme.

    \item We say that a variety is a curve (resp.\ surface, resp.\ $n$--fold) if it has dimension one (resp.\ two, resp.\ $n$).

    \item Usually, $X$ will denote a smooth $n$--fold, $V$ will denote an arbitrary $d$--fold, $S$ will denote a surface, $C$ will denote a curve, and $x$ will denote a closed point.

    \item On a variety $V$, by \emph{divisor} we mean a Cartier divisor. That is, if we denote by $\cK_V$ the sheaf of rational functions on $V$, and by $\cK_V^*$ resp.\ $\cO_V^*$ the sheaf of units in $\cK_V$ resp.\ $\cO_V$, then a divisor is a global section of $\cK^*_V/\cO_V^*$. We denote by $\Div(V)\coloneqq H^0\left(V,\cK^*_V/\cO_V^*\right)$ the group of divisors on $V$.
    
    \item By \emph{$\bQ$--divisor}, we mean a $\bQ$--Cartier $\bQ$--divisor, i.e.,\ an element of $\Div(V)\otimes\bQ$. Furthermore, a $\bQ$--divisor is called \emph{integral} if it is in the image of the natural map $\Div(V)\to \Div(V)\otimes\bQ$. Note that if $V$ is not normal, then $\Div(V)$ might have torsion. 

    \item A divisor is effective if it has a representative $\{(U_i,f_i)\}_i$ such that $f_i$ is regular on $U_i$ for all $i$. A $\bQ$--divisor is effective if it is a non-negative rational multiple of an effective divisor.

    \item For a $\bQ$--divisor $A$, we denote by $|A|_\bQ$ its $\bQ$--linear system, i.e.,\
    \begin{align*}
        |A|_{\bQ}\coloneqq\{D\in \Div(V)\otimes\bQ\mid D\text{ effective, }D\sim_{\bQ} A\}.
    \end{align*}
    By definition, $D\sim_{\bQ} A$ means that there exists $l>0$ such that $lD$ and $lA$ are integral and linearily equivalent.

    \item For an effective divisor $D$ on $V$, we denote by $\Supp(D)$ its support, i.e.,\ the set of points $\eta$ such that a local equation of $D$ around $\eta$ is not a unit in $\cO_{V,\eta}$. If $D$ is an effective $\bQ$--divisor, we write $D=(1/a)D'$ for an integer $a$ and an integral effective divisor $D'$ and define $\Supp(D)\coloneqq\Supp(D')$. Note that this is well--defined even if $\Div(V)$ has torsion.

    \item By \cite[Definition 1.3.3.(iv)]{Lazarsfeld_Positivity_in_algebraic_geometry_I}, we can pull back $\bQ$--divisors along blow-ups, and along closed immersions if the image meets the support properly.

    \item For a line bundle $\cL$ and a curve $C$ on a variety $V$, we denote by
    \begin{align*}
        \cL.C\coloneqq \deg \nu^*(\cL|_C)
    \end{align*}
    the intersection number of $\cL$ and $C$, where $\nu\colon C^\nu\to C$ is the normalization of $C$. We extend this to divisors with the formula
    \begin{align*}
        D.C\coloneqq\cO_V(D).C
    \end{align*}
    and to $\bQ$--divisors by linearity.

    \item If $V$ is a local complete intersection (lci), then its canoncial sheaf $\omega_V$ is a line bundle, and we fix a divisor $K_V$ such that $\omega_V=\cO_V(K_V)$.

    \item By \cite[Proposition 29.4.8]{Vakil_The_rising_sea}, if $D$ is an effective divisor on a smooth variety $X$, it is in particular a local complete intersection, and we have
    \begin{align*}
        \omega_D=\omega_X(D)|_D.
    \end{align*}
    In contexts where we can work up to linear equivalence, e.g.,\ when working with intersection numbers, we will use this in additive form implicitly as $K_D=(K_X+D)|_D$, by a slight abuse of notation.
    
\end{itemize}

\subsection{Acknowledgements}\label{subsec:acknowledgements}

The author would like to thank his advisor Zsolt Patakfalvi for numerous fruitful discussions and continuous guidance throughout this project. He would also like to thank Jefferson Baudin, Raymond Cheng, Stefano Filipazzi, and Alapan Mukhopadhyay for useful discussions and/or comments on the content of this article. The author was partly supported by the  project grant \#200020B/192035 from the Swiss National Science
Foundation (FNS), as well as by the ERC starting grant \#804334.

\section{Preliminaries}\label{sec:prelims}
In this section, we recall some definitions and basic properties of multiplicities of varieties and vanishing orders of divisors at points. We also gather some useful properties of the intersection numbers of adjoint divisors with curves. Finally, we define Seshadri constants.

\subsection{Multiplicity and asymptotic order of vanishing}\label{subsec:multiplicity_and_order}
We start with a fundamental notion in this respect, the \emph{Hilbert--Samuel multiplicity}.
\begin{definition}[\protect{\cite[Chapter 4.3]{Fulton_Intersection_theory}}]\label{def:Hilbert-Samuel-multiplicity}
    Let $V$ be a pure--dimensional scheme of finite type over $k$, and let $\eta$ be a scheme point of $V$ of codimension $c$. The \emph{Hilbert--Samuel multiplicity} of $V$ at $\eta$, denoted $\mult_\eta V$, is defined by the equation
    \begin{align*}
        \length\left(\factor{\cO_{V,\eta}}{\fm_{V,\eta}^r}\right)=\frac{\mult_\eta V}{c!}r^c+O(r^{c-1})
    \end{align*}
    as $r$ tends to infinity. We will often call this simply the multiplicity of $V$ at $\eta$. For an integral subscheme $Y$ of $V$ with generic point $\eta$, we will also write $\mult_YV=\mult_\eta V$.
\end{definition}

We will also need the following notion of the order of an effective divisor at a point.

\begin{definition}\label{def:order_divisor}
    Let $V$ be a variety, let $D$ be an effective divisor on $V$ and let $\eta\in V$ be a scheme point. Let $f\in\cO_{V,\eta}$ be a local equation of $D$ around $\eta$. The vanishing order of $D$ at $\eta$, denoted $\ord_\eta D$, is defined as
    \begin{align*}
        \ord_\eta D\coloneqq \max \left\{r\in\bZ_{\geq 0}\mid f\in\fm_{V,\eta}^r\right\}.
    \end{align*}
    For a subvariety $Y$ of $V$ with generic point $\eta$, we will also write $\ord_YD=\ord_\eta D$.
\end{definition}

\begin{remark}\label{rem:properties_order}
    \begin{itemize}
        \item We have $\ord_\eta D>0$ if and only if $\eta$ is in the support of $D$.
        \item If $V$ is regular at $\eta $, then the Hilbert--Samuel multiplicity $\mult_\eta D$ of $D$ at $\eta $ agrees with the order $\ord_\eta  D$, see, e.g.,\ \cite[Exercise 14.5]{Matsumura_Commuatative_ring_theory}. However, this can fail if $V$ is singular at $\eta $. For example, if we take $V$ to be the cusp $v^3=u^2$ and $D$ the divisor cut out by $u$, then if $x$ is the origin, we have $\mult_xD=2$ but $\ord_xD=1$.
        \item If $V$ is regular at $\eta $, then we have $\ord_\eta (nD)=n\ord_\eta D$ for all $n\in\bZ_{>0}$, so we can extend the definition to effective $\bQ$--divisors in a straightforward way. However, this may fail at singular points: with the same notation as in the previous point, note that $\ord_x(2D)=3$ while $\ord_xD=1$. So, at singular points, we need to be more careful.
        \item Note that we always have $\ord_\eta (D+D')\geq\ord_\eta  D+\ord_\eta  D'$. In particular, the sequence $\{\ord_\eta (lD)\}_{l>0}$ is super-additive.
    \end{itemize}
\end{remark}

As foreshadowed in \autoref{subsec:strategy_and_main_results}, we need to extend the definition of the order of vanishing to the case of effective $\bQ$--divisors. A concept that goes in this direction is \cite[Definition 2.2]{Ein_Lazarsfeld_Musta_Nakamaye_Popa_Asymptotic_invariants_of_base_loci}. However, in that definition, the order of vanishing is minimized in a given linear system, but to make lower bounds like \eqref{eq:lb_intersection_intro} as good as possible, we rather want to maximize it. Therefore we introduce the following definition, which is essentially just the asymptotic Samuel function (\cite[Definition 1.1]{Bravo_Encinas_Guillan-Rial_On_some_properties_of_the_asymptotic_Samuel_function}) applied to a local equation of the divisor.

\begin{definition}\label{def:order_Q-divisor}
    Let $V$ be a variety, let $D$ be an effective $\bQ$--divisor on $V$ and let $\eta \in V$ be a scheme point. Write $D=(1/a)D'$ where $a$ is a positive integer and $D'$ is an integral divisor. By \autoref{rem:properties_order}, the sequence $\{\ord_{\eta }(lD')\}_{l>0}$ is super--additive. In particular, by Fekete's lemma (\cite{Fekete_Uber_die_Verteilung_der_Wurzeln_bei_gewissen_algebraischen_Gleichungen_mit_ganzzahligen_Koeffizienten}), we may define
    \begin{align*}
        \Ard{\eta }{D}\coloneqq \frac{1}{a}\lim_{l\to \infty}\frac{\ord_\eta (lD')}{l}.
    \end{align*}
    For a subvariety $Y$ of $V$ with generic point $\eta $, we will also write $\Ard{Y}{D}=\Ard{\eta }{D}$.
\end{definition}

\begin{remark}\label{rem:properties_asymptotic_order}
    \begin{itemize}
        \item It is straightforward to check that the asymptotic order is well--defined: suppose that we have $D=(1/a)D'=(1/a')D''$. Then there exists a positive integer $l_0$ such that $l_0a'D'=l_0aD''$ inside $\Div(V)$. Hence
        \begin{align*}
            \frac{1}{a}\lim_{l\to\infty}\frac{\ord_\eta(lD')}{l}=\frac{1}{a}\lim_{l\to\infty}\frac{\ord_\eta(ll_0a'D')}{ll_0a'}=\frac{1}{a'}\lim_{l\to\infty}\frac{\ord_\eta(ll_0aD'')}{ll_0a}=\frac{1}{a'}\lim_{l\to\infty}\frac{\ord_\eta(lD'')}{l}
        \end{align*}
        \item We have $\Ard{\eta }{D}>0$ if and only if $\eta $ is in the support of $D$. Furthermore, by \cite[Equation 1.11]{Bravo_Encinas_Guillan-Rial_On_some_properties_of_the_asymptotic_Samuel_function}, $\Ard{\eta }{D}$ is always finite and rational.
        \item Fekete's lemma further implies that we have
        \begin{align*}
            \Ard{\eta }{D}=\sup\left\{\left.\frac{\ord_\eta (D')}{a}\ \right|\ D'\in\Div(V),\ a\in\bZ_{>0},\ D=(1/a)D'\right\}.
        \end{align*}
        \item If $V$ is regular at $\eta $ and $D=c D'$ where $c\geq 0$ is rational and $D'$ is integral, we have $\Ard{\eta }{D}=c\ord_\eta D'$. In particular, $\Ard{\eta}{no}$ and $\ord_\eta$ agree on integral divisors. However, if $\eta $ is a singular point, this might not hold, as, e.g.,\ in the example of \autoref{rem:properties_order}, one can show that $\Ard{x}{D}=3/2$.
        \item We trivially have $\Ard{\eta }{nD}=n\Ard{\eta }{D}$ for all $n\in\bZ_{>0}$, even if $V$ is singular at $\eta $.
    \end{itemize}
\end{remark}

We will now prove some basic properties about multiplicity and the asymptotic order. 

\begin{lemma}\label{lem:order_transitive_restriction}
    Let $V$ be a variety, let $x\in V$ be a regular closed point, and let $Y\subset V$ be a prime divisor containing $x$. Let $D$ be an effective $\bQ$--divisor on $V$. Then
    \begin{align*}
        \Ard{x}{D}\geq\mult_xY\cdot\Ard{Y}{D}
    \end{align*}
\end{lemma}

\begin{proof}
    Let $g\in\cO_{V,x}$ be a local equation of $Y$, and write $D=(1/a)D'$ for a positive integer $a$ and an integral divisor $D'$. Let $f\in \cO_{V,x}$ be a local equation of $D'$. Let $\eta$ be the generic point of $Y$. As $x$ is a regular point, $\eta$ is regular as well, and thus
    \begin{align*}
        \Ard{x}{D}=\frac{1}{a}\ord_x(D'),\quad \Ard{Y}{D}=\frac{1}{a}\ord_Y(D').
    \end{align*}
    Furthermore, $\cO_{V,x}$ is a UFD, and hence we have
    \begin{align*}
        f=g^r\cdot h
    \end{align*}
    where $h\in\cO_{V,x}$ and $r=\ord_Y(D')$. As $m\coloneqq\mult_x Y=\ord_x Y$ by \autoref{rem:properties_order}, we have $g\in\fm_{V,x}^m$, and thus $f=g^rh\in\fm_{V,x}^{mr}$. Therefore, we obtain $\ord_x(D')\geq mr$, which after dividing by $a$ concludes the proof.
\end{proof}

\begin{remark}\label{rem:counterexample_order_semicont}
    It is natural to ask what happens when removing the smoothness assumptions in \autoref{lem:order_transitive_restriction}. One might think that at least the function $x\mapsto\Ard{x}{D}$ can only decrease under generization, but this already fails for the quadratic cone $V=V(w^2-uv)\subseteq\bA^3$: let $x\in V$ be the origin, $Y=V(u,w)$ be one line of the ruling and $D$ be the divisor cut out by $u$. Using \cite[Theorem 2.4]{Bravo_Encinas_Guillan-Rial_On_some_properties_of_the_asymptotic_Samuel_function}, it's straightforward to compute that $\Ard{x}{D}=1$. On the other hand, as $V$ is regular at the generic point of $Y$, the asymptotic order agrees with the usual order of vanishing by \autoref{rem:properties_asymptotic_order}, and one can see that $\Ard{x}{D}=2$. In conclusion, we have
    \begin{align*}
        \Ard{x}{D}=1<2=\Ard{Y}{D}.
    \end{align*}
    A tempting guess would then be that at least the function $x\mapsto \mult_x V\Ard{x}{D}$ can only decrease under generization. But in fact also this fails strongly. The following counterexample was kindly communicated to the author by Ana Bravo, Santiago Encinas and Javier Guillan--Riál: Consider $V=V(w^{n+1}-uv)\subseteq\bA^3$. Note that $V$ has an $A_n$--singularity at the origin $x\in V$, and in particular $\mult_x V=2$. Let $Y=V(u,w)$ and let $D$ be the effective divisor cut out by $u$. By the same computation as for the quadratic cone, one can then see that
    \begin{align*}
        \Ard{x}{D}=1<n+1=\Ard{Y}{D}.
    \end{align*}
    Hence, it is in general impossible to bound $\Ard{x}{D}/\Ard{Y}{D}$ from below only in terms of $\mult_x V$. The only positive result in this direction seems to be \cite[Theorem 3.1]{Bravo_Encinas_Guillan-Rial_On_some_properties_of_the_asymptotic_Samuel_function}, which says that if $\mult_xV=\mult_YV$, then $\Ard{x}{D}\geq\Ard{Y}{D}$.
\end{remark}

The following lemma is a slight generalization of the usual lower bound of intersection numbers by the product of multiplicities.

\begin{lemma}\label{lem:intersection_lb_multiplicities}
    Let $V$ be a variety and let $D$ be an effective $\bQ$--divisor on $V$. Let $x\in V$ be a closed point and $C$ a curve through $x$ such that $C\not\subseteq\Supp D$. Then we have
    \begin{align*}
        D.C\geq\Ard{x}{D}\cdot\mult_x C.
    \end{align*}
\end{lemma}

\begin{proof}
    It suffices to prove that if $D$ is an integral divisor, then
    \begin{align*}
        D.C\geq\ord_x(D)\cdot\mult_x C;
    \end{align*}
    the assertion then follows by passing to a limit. Now we follow \cite[Example 2.2.3]{Murayama_Seshadri_constants_and_fujitas_conjecture_via_positive_characteristic_methods}. Let $f\in \cO_{V,x}$ be a local equation of $D$. Viewing $D|_C$ as a locally principal closed subscheme of $C$, we then have
    \begin{align*}
        D.C\geq\length_{\cO_{C,x}}(\cO_{D|_C,x})=\length\left(\factor{\cO_{C,x}}{f\cdot\cO_{C,x}}\right),
    \end{align*}
    where we used \cite[Definition 15.29]{Görtz_Wedhorn_Algebraic_geometry_I}. Now, if we denote $r=\ord_x(D)$, then we have $f\in\fm_{V,x}^r$ and thus $f\cdot\cO_{C,x}\subseteq\fm_{C,x}^r$. To conclude, we obtain by \cite[Theorem 14.10]{Matsumura_Commuatative_ring_theory} that
    \begin{align*}
        D.C\geq\length\left(\factor{\cO_{C,x}}{f\cdot\cO_{C,x}}\right)\geq r\cdot\mult_x C,
    \end{align*}
    so we are done.
\end{proof}

The above lemma states that we can bound from below the ratio $D.C/\mult_xC$, at least if $C\not\subseteq\Supp D$. As obtaining such bounds is exactly what we are after, this suggests that we should optimize the number $\Ard{x}{D}$ on the right hand side. This is what we do in the following lemma; it is a well known consequence of asymptotic Riemann--Roch.

\begin{lemma}\label{lem:div_high_mult_irred}
    Let $X$ be a variety of dimension $n$, let $A$ be an ample divisor, $V\subseteq X$ a subvariety of dimension $d>0$, and $x\in V$ a closed point. Then for all $\delta>0$, there exists $D\in|A|_{\bQ}$ with $V\not\subseteq\Supp D$ and
    \begin{align*}
        \Ard{x}{D|_V}>\left(\frac{A^d.V}{\mult_xV}\right)^{1/d}-\delta.
    \end{align*}
\end{lemma}

\begin{proof}
    The following proof is a simplified version of the proof of \cite[Lemma 3.2]{Cascini_Tanaka_Xu_On_base_point_freeness_in_positive_characteristic}. Let $\sigma$ be a positive real number strictly smaller than $\sigma_0\coloneqq\left(\frac{A^d.V}{\mult_xV}\right)^{1/d}$. Let $V^{(l)}$ be the closed subscheme of $V$ cut out by $\cI_{x}^{\lfloor\sigma l\rfloor}$, where $\cI_{x}$ is the ideal sheaf of $\{x\}\subset V$. Consider the restriction map
    \begin{align*}
        H^0(V,\cO_X(lA)|_V)\to H^0\left(V^{(l)},\cO_X(lA)|_{V^{(l)}}\right).
    \end{align*}
    We want to argue that for $l$ large enough, this map has a non-trivial kernel. On the one hand, notice that
    \begin{align*}
        h^0(V,\cO_X(lA)|_V)=\frac{A^d.V}{d!}l^d+(\text{lower order terms}),
    \end{align*}
    by \cite[Example 1.2.19]{Lazarsfeld_Positivity_in_algebraic_geometry_I}. On the other hand, since $\cO_V/\cI_{x}^{\lfloor\sigma l\rfloor}$ is a skyscraper sheaf concentrated at $x$ and $\cO_X(lA)|_V$ is locally free, we have an isomorphism $\cO_X(lA)|_{V^{(l)}}\cong \cO_{V,x}/\fm_{V,x}^{\lfloor\sigma l\rfloor}$, and thus
    \begin{align*}
        h^0\left((V^{(l)},\cO_X(lA)|_{V^{(l)}}\right)&= \length(\cO_{V,x}/\fm_{V,x}^{\lfloor\sigma l\rfloor})\\
        &=\frac{\mult_x V}{d!}(\sigma l)^d+(\text{lower order terms}).
    \end{align*}
    Now, since we chose $\sigma <\sigma_0$, it follows that $h^0(V,\cO_X(lA)|_V)> h^0\left(V^{(l)},\cO_X(lA)|_{V^{(l)}}\right)$ for $l\gg 0$. Hence, for $l$ large enough, the kernel of the restriction is non-empty. Furthermore, by Serre vanishing, we find that for $l\gg 0$, there exists $s\in H^0(X,\cO_X(lA))$ with $s|_V\neq 0$ and $s|_{V^{(l)}}=0$.

    Set $D\coloneqq\frac{1}{l}\operatorname{div}(s)$. Then, by construction, we have
    \begin{align*}
        \Ard{x}{D|_V}\geq\frac{\ord_x(lD|_V)}{l}\geq\frac{\lfloor\sigma l\rfloor}{l}
    \end{align*}
    So, choosing $\sigma$ close enough to $\sigma_0$ and $l$ large enough so that $\lfloor\sigma l\rfloor>(\sigma_0-\delta)l$, we obtain $\Ard{x}{D|_Y}>\sigma_0-\delta$.
\end{proof}

Finally, we will also need the following geometric characterization of the asymptotic order.

\begin{lemma}\label{lem:blow-up_formula_restricted_divisor}
    Let $V$ be a variety and let $x\in V$ be a closed point. Let $D$ be an effective $\bQ$--divisor and let $s<\Ard{x}{D}$ be a rational number. Finally, consider the blow--up $\pi\colon\wt{V}\to V$ and denote by $E$ the exceptional divisor. Then $\pi^*D-sE$ is an effective $\bQ$--divisor on $\wt{V}$.
\end{lemma}

\begin{proof}
    By passing to a suitable multiple of $D$, it suffices to prove that if $D$ is an integral effective divisor and $r=\ord_x D$, then $\pi^*D-rE$ is effective. The statement is clearly local around $x$, so we may assume that $V$ is affine and that $D$ is cut out by a global equation $f\in \Gamma(V,\cO_V)$ satisfying $f\in \fm_x^r$. Fix an embedding $V\subseteq\bA^N$ and assume without loss of generality that $x=0$. Let $I\subseteq k[x_1,\ldots,x_N]$ be the ideal of $V$ inside $\bA^N$. Let $\{j_i\colon U_i\to\bA^N\}_i$ be the standard open cover of $\wt{\bA^N}\subseteq \bA^N\times\bP^{N-1}$, i.e.,\
    \begin{align*}
        j_i\colon\bA^N&\to\wt{\bA^N}\subseteq \bA^{N}\times\bP^{N-1}\\
        (a_1,\ldots,a_N)&\mapsto \left((a_ia_1,\ldots,a_ia_{i-1},a_i,a_ia_{i+1},\ldots,a_ia_N),[a_1:\ldots:a_{i-1}:1:a_{i+1}:\ldots:a_N]\right).
    \end{align*}
    It is straightforward to compute that the exceptional divisor in $\wt{V}$ can be written as $E=\{(\wt{V}\cap U_i,x_i)\}$. We can view $D$ as being cut out by a polynomial $f\in k[x_1,\ldots,x_N]$ such that $f\in \fm_{x}^r+I$, so that we can write $f=g+h$ with $g\in \fm_{x}^r$ and $h\in I$. Then we have
    \begin{align*}
        j_i^{\sharp}(f)=f(x_ix_1,\ldots,x_ix_{i-1},x_i,x_ix_{i+1},\ldots,x_ix_N)=x_i^rg_i+h_i
    \end{align*}
    for some $g_i\in k[x_1,\ldots,x_N]$ and $h_i$ an element of the ideal of $\wt{X}\cap U_i$ in $U_i$. Hence, we can write $\pi^*D$ as $\{(\wt{X}\cap U_i,x_i^rg_i)\}$, and thus
    \begin{align*}
        \pi^*D-rE=\{(\wt{X}\cap U_i,g_i)\}
    \end{align*}
    is effective.
\end{proof}

\subsection{Intersecting adjoint divisors and curves}\label{subsec:intersecting_adjoints}
In this subsection, we study some intersection--theoretical properties of adjoint divisors. It is well known that if $X$ is a smooth variety of dimension $n$ and $A$ is an ample divisor on $X$, then $K_X+(n+1)A$ is nef. This is often implicitly stated as part of the Cone Theorem, but it is actually already a consequence of Mori's Bend--and--Break technique. We will use a remarkably general version of this technique, which can be found in \cite[Section II.5]{Kollar_Rational_curves_on_algebraic_varieties}.
\begin{lemma}\label{lem:bend-and-break}
    Let $V$ be an lci variety of dimension $d$ and let $M$ be a nef divisor on $V$. Assume that $C$ is a curve in $V$ that intersects the regular locus and such that $(K_V+2dM).C<0$. Then $C$ is in the exceptional locus $\bE(M)$ of $M$ (i.e.,\ the union of all subvarieties on which $M$ is not big).
\end{lemma}

\begin{proof}
    As $(K_V+2dM).C<0$ and $M$ is nef, we must have $K_V.C<0$. By Remark 5.15 in \emph{loc.cit.}, Theorem 5.8 applies to the map $C^{\nu}\to V$, so in particular there exists a rational curve $L_x\subseteq V$ through any point $x\in C$ such that
    \begin{align*}
        M.L_x\leq 2d\frac{M.C}{-K_V.C}.
    \end{align*}
    Note that as $(K_V+2dM).C<0$, the right hand side is strictly smaller than $1$, and as $M.L_x$ is a non-negative integer, we obtain $M.L_x=0$. In particular, we have $x\in\bE(M)$, and as this holds for all $x\in C$, we conclude $C\subseteq\bE(M)$.
\end{proof}

As a corollary, we obtain the following useful fact.
\begin{corollary}\label{cor:K_X+2(dimX)A_nef}
    Let $V$ be an lci variety of dimension $d$, let $A$ be an ample divisor on $V$ and $C\subseteq V$ a curve intersecting the regular locus of $V$. Then $(K_V+2dA).C\geq 0$. In particular, if $V$ has isolated singularities, then $K_V+2dA$ is nef.
\end{corollary}

\begin{remark}\label{rem:K_X+2(dimX)A_nef_surface}
    For an lci surface $S$, we actually have $(K_S+3A).C\geq 0$ for curves intersecting the regular locus. Indeed, denote by $\psi\colon\wht{S}\to S^{\nu}$ the minimal resolution of the normalization of $S$ (see \cite[Theorem 4-6-2]{Matsuki_Introduction_to_the_Mori_program}), and by $\wht{C}$ the strict transform of $C$ in $\wht{S}$. By the proof of \autoref{lem:adjunction_inequality_surfaces_auxilliary}, we have 
    \begin{align*}
        K_S.C\geq K_{\wht{S}}.\wht{C}.
    \end{align*}
    Therefore
    \begin{align*}
        (K_S+3A).C\geq (K_{\wht{S}}+3\psi^*\nu^*A).\wht{C}\geq 0,
    \end{align*}
    where the last inequality follows from the Cone Theorem on $\wht{S}$ (or, e.g.,\ \cite{Jovinelly_Lehmann_Riedl_Optimal_bounds_in_Bend-and-Break}).
\end{remark}

We will conclude this subsection by proving analogs of \eqref{eq:adjunction_intro} in the cases we need.

\begin{lemma}\label{lem:Seshadri_hypersurface_high_mult}
    Let $X$ be a smooth variety of dimension $n$, let $x\in X$ be a closed point and $A$ and ample divisor on $X$. Let $V\subseteq X$ be a prime divisor and $C\subseteq V$ an integral curve through $x$, intersecting the regular locus of $V$. Then
    \begin{align*}
        (K_X+V+2(n-1)A).C\geq(\mult_x V-n+1)\mult_x C.
    \end{align*}
\end{lemma}

\begin{proof}
    Let $\pi\colon\wt{X}\to X$ be the blow-up of $X$ at $x$, denote by $E$ the exceptional divisor and by $\wt{V}$ resp.\ $\wt{C}$ the strict transforms of $V$ resp.\ $C$. In addition, to lighten the notation, denote $\mu_V=\mult_xV$ and $\mu_C=\mult_x C$. By the projection formula, we have
    \begin{align*}
        (K_X+V+2(n-1)A).C=\pi^*(K_X+V+2(n-1)A).\widetilde{C}.
    \end{align*}
    Now note that $\pi^*K_X=K_{\widetilde{X}}-(n-1)E$ (\cite[Exercise II.8.5]{Hartshorne_Algebraic_geometry}) and $\pi^*V=\widetilde{V}+\mu_VE$ (\cite[Corollary 6.7.1]{Fulton_Intersection_theory}). Hence
    \begin{align*}
        (K_X+V+2(n-1)A).C=(K_{\widetilde{X}}+\widetilde{V}+2(n-1)\pi^*A).\widetilde{C}+(\mu_V-n+1)E.\widetilde{C}.
    \end{align*}
    Now by \cite[Corollary 6.7.1]{Fulton_Intersection_theory}, if $\ell\subseteq E\cong\bP^{n-1}$ is a line, then
    \begin{align*}
        E.\widetilde{C}=-\mu_CE.\ell=\mu_C,
    \end{align*}
    where we used $-E.\ell=\deg\cO_{E}(1)|_{\ell}=1$. Therefore, we are left with proving that $(K_{\widetilde{X}}+\widetilde{V}+2(n-1)\pi^*A).\widetilde{C}$ is non-negative.

    To do so, note that we have
    \begin{align*}
        (K_{\widetilde{X}}+\widetilde{V}+2(n-1)\pi^*A).\widetilde{C}&=(K_{\widetilde{X}}+\widetilde{V}+2(n-1)\pi^*A)|_{\widetilde{V}}.\widetilde{C}\\
        &=(K_{\widetilde{V}}+2(n-1)(\pi^*A)|_{\widetilde{V}}).\widetilde{C}.
    \end{align*}
    Now denote $M=(\pi^*A)|_{\widetilde{V}}$, and note that $M$ is nef, because pull-backs preserve nefness. Assume by contradiction that $(K_{\widetilde{V}}+2(n-1)(\pi^*A)|_{\widetilde{V}}).\widetilde{C}<0$, then by \autoref{lem:bend-and-break}, we obtain $\widetilde{C}\subseteq \bE(M)$. But as $A$ is ample, we must have $\bE(M)\subseteq E$, and thus we obtain $\wt{C}\subseteq E$, which is a contradiction. Hence, we must have $(K_{\widetilde{V}}+2(n-1)(\pi^*A)|_{\widetilde{V}}).\widetilde{C}\geq 0$, which concludes the proof.
\end{proof}

\begin{remark}\label{rem:bend-and-break}
    By adjunction, \autoref{lem:bend-and-break} shows, in particular, that
    \begin{align*}
        (K_V+2(n-1)A|_V).C\geq (\mult_x V-n+1)\mult_x C.
    \end{align*}
    If $\mult_x V<n$, then this bound is useless, but it grows linearly as $\mult_x V$ tends to infinity. This should be regarded as an example of the general phenomenon that singularities contribute positively to $K_X$.
\end{remark}

Finally, in order to prove \autoref{thmB}, we will need to work on prime divisors $S$ in a smooth threefold $X$. These are in particular local complete intersections, and the following is a version of \eqref{eq:adjunction_intro} in this more general setting.

\begin{lemma}\label{lem:adjunction_inequality_surfaces}
    Let $S$ be an lci surface, $x\in S$ a closed regular point and $C\subset S$ an integral curve through $x$. Let $D$ be an effective $\bQ$--divisor on $S$ such that $\Ard{C}{D}=1$. Then
    \begin{align*}
        (K_S+D).C\geq(\Ard{x}{D}-1)\mult_x C-2.
    \end{align*}
\end{lemma}

\begin{proof}
    To simplify the notation, denote $r\coloneqq\Ard{x}{D}$ and $\mu\coloneqq\mult_xC$, and let $s<r$ be a rational number. Let $\pi\colon\wt{S}\to S$ be the blow-up of $S$ at $x$, denote by $E$ the exceptional divisor and by $\wt{C}$ the strict transform of $C$. Then we have
    \begin{align*}
        (K_S+D).C&=\pi^*(K_S+D).\wt{C}\\
        &=(K_{\wt{S}}-E+\pi^*D).\wt{C}\\
        &=(K_{\wt{S}}+\underbrace{\pi^*D-sE}_{D'\coloneqq})+(s-1)E.\wt{C}
    \end{align*}
    Now note that $D'$ is an effective $\bQ$--divisor on $\wt{S}$ due to \autoref{lem:blow-up_formula_restricted_divisor}, and that $\Ard{\wt{C}}{D'}=1$ (because the local equations of $D$ and $D'$ agree in $\cO_{S,C}\cong \cO_{\wt{S},\wt{C}}$). Furthermore, $\wt{S}$ is still a local complete intersection, as we blew up a regular point, and $\wt{C}$ intersects the regular locus of $\wt{S}$. Hence, by \autoref{lem:adjunction_inequality_surfaces_auxilliary}, we have $(K_{\wt{S}}+D').\wt{C}\geq -2$. As also $E.\wt{C}=\mu$ by \cite[Corollary 6.7.1]{Fulton_Intersection_theory}, we obtain
    \begin{align*}
        (K_S+D).C\geq (s-1)\mu-2.
    \end{align*}
    As this holds for every rational number $s<r$, we obtain the result.
\end{proof}

\begin{lemma}\label{lem:adjunction_inequality_surfaces_auxilliary}
    Let $S$ be an lci surface, let $C$ be an integral curve intersecting the regular locus of $S$ and $D$ an effective $\bQ$--divisor with $\Ard{C}{D}=1$. Then
    \begin{align*}
        (K_S+D).C\geq -2.
    \end{align*}
\end{lemma}

\begin{proof}
    Let $\nu\colon \bar{S}\to S$ be the normalization of $S$ and $\psi\colon\widehat{S}\to \bar{S}$ the minimal resolution of $\bar{S}$ (\cite[Theorem 4-6-2]{Matsuki_Introduction_to_the_Mori_program}). As $C$ intersects the regular locus, we can form the strict transforms $\bar{C}\subset\bar{S}$ and $\widehat{C}\subset\widehat{S}$ of $C$. By \cite[5.7, p.191]{Kollar_Singularities_of_the_minimal_model_program}, if $\bar{B}$ denotes the Weil divisor cut out by the conductor ideal in $\bar{S}$, we have a natural isomorphism
    \begin{align*}
        \omega_{\bar{S}}\cong(\nu^*\omega_S)(-\bar{B}).
    \end{align*}
    This means that the difference between the Weil divisors $K_{\bar{S}}$ and $\nu^*K_S-\bar{B}$ is a principal Cartier divisor. In particular, we then have
    \begin{align*}
        K_{\bar{S}}.\bar{C}=(\nu^*K_S-\bar{B}).\bar{C}
    \end{align*}
    in the sense of Mumford's intersection theory (see \cite[Examples 7.1.16 and 8.3.11]{Fulton_Intersection_theory}). Now by \cite[Theorem 4-6-2]{Matsuki_Introduction_to_the_Mori_program}, the Mumford pull-back of $K_{\bar{S}}$ along $\psi$ can be written as
    \begin{align*}
        \psi^*K_{\bar{S}}=K_{\widehat{S}}+\sum_i b_iE_i,
    \end{align*}
    where the $E_i$'s are the exceptional curves and $b_i\geq 0$ for all $i$. Hence, we have
    \begin{align*}
        K_{\bar{S}}.\bar{C}=\psi^*K_{\bar{S}}.\widehat{C}=(K_{\widehat{S}}+\sum_i b_iE_i).\widehat{C}\geq K_{\widehat{S}}.\widehat{C}.
    \end{align*}
    On the other hand, we have
    \begin{align*}
        K_S.C=(\nu^*K_S).\bar{C}\geq (\nu^*K_S-\bar{B}).\bar{C}=K_{\bar{S}}.\bar{C},
    \end{align*}
    where we used the projection formula and \autoref{rem:mumford_proper_intersection}. By adding $D$, we obtain
    \begin{align*}
        (K_S+D).C\geq (K_{\widehat{S}}+\psi^*\nu^*D).\widehat{C}.
    \end{align*}
    Finally, note that $\psi^*\nu^*D$ is an effective $\bQ$--divisor with $\Ard{\widehat{C}}{\psi^*\nu^*D}=\Ard{C}{D}=1$, and thus we can write
    \begin{align*}
        \psi^*\nu^*D=\widehat{C}+D',
    \end{align*}
    where $D'$ is an effective $\bQ$--divisor that intersects properly with $\widehat{C}\not\subseteq \Supp D'$. Hence, we have
    \begin{align*}
        (K_S+D).C\geq (K_{\widehat{S}}+\psi^*\nu^*D).\widehat{C}=(K_{\widehat{S}}+\widehat{C}+D').\widehat{C}\geq(K_{\widehat{S}}+\widehat{C}).\widehat{C}\geq -2
    \end{align*}
    by the arithmetic adjunction formula on $\widehat{S}$.
\end{proof}

\begin{remark}\label{rem:mumford_proper_intersection}
    Note that in Mumford's intersection theory on a normal surface $S$, we still have $C.C'\geq 0$ for all curves $C\neq C'$. Indeed, if $\psi\colon\widehat{S}\to S$ denotes the minimal resolution with exceptional curves $E_i$, then the matrix $M=(E_i.E_j)_{ij}$ is negative definite by \cite[Theorem 4-6-1]{Matsuki_Introduction_to_the_Mori_program}. Let us define the vectors $\gamma$ and $\gamma'$ by
    \begin{align*}
        \gamma\coloneqq (\widehat{C}.E_i)_i,\quad \gamma'\coloneqq (\widehat{C'}.E_i)_i,
    \end{align*}
    then it is straightforward to compute that
    \begin{align*}
        C.C'=\widehat{C}.\widehat{C'}-\gamma^{\intercal}M^{-1}\gamma'.
    \end{align*}
    Now note that $-M$ is a Stieltjes matrix, so $-M^{-1}$ has non-negative entries, as do $\gamma$ and $\gamma'$. Hence $-\gamma^{\intercal}M^{-1}\gamma'\geq 0$, so that $C.C'\geq 0$.
\end{remark}

\subsection{Seshadri constants}\label{subsec:Seshadri_constants}
Let us start with the definition of the Seshadri constant, which we will be using in this article.

\begin{definition}\label{def:Seshadri-constants}
    Let $V$ be a variety, $x\in V$ a closed point and $L$ a $\bQ$--divisor on $V$ such that $L.C\geq 0$ for all curves $C\subset V$ through $x$. The \emph{intersectional Seshadri-constant of $L$ at $x$}, denoted $\epsint(L;x)$, is defined by the expression
    \begin{align*}
        \epsint(L;x)\coloneqq\inf_{x\in \underset{\text{curve}}{C}\subset V}\frac{L.C}{\mult_x C}.
    \end{align*}
    If there exists a curve $C\subset V$ such that
    \begin{align*}
        \epsint(L;x)=\frac{L.C}{\mult_x C},
    \end{align*}
    we will call $C$ a \emph{Seshadri curve}.
\end{definition}

\begin{remark}\label{rem:def_Seshadri_nef_case}
    Note that if $L$ is nef, the intersectional Seshadri constant agrees with the usual one: in this case, if $\pi\colon\wt{V}\to V$ denotes the blow-up of $V$ at $x$ with exceptional divisor $E$, we have
    \begin{align*}
        \epsint(L;x)=\sup\{t\in\bQ_{\geq 0}\mid \pi^*L-tE\text{ nef}\},
    \end{align*}
    see \cite[Proposition 5.1.5]{Lazarsfeld_Positivity_in_algebraic_geometry_I}.
\end{remark}

The reason for using \autoref{def:Seshadri-constants} in the slightly more general case where we only suppose that $L$ is non--negative along curves through $x$ is because it is well adapted to working with prime divisors $V$ in a smooth ambient variety $X$ in which $x$ is a regular closed point. On the one hand, due to \autoref{cor:K_X+2(dimX)A_nef}, the hypothesis will be naturally satisfied for certain adjoint divisors. On the other hand, \autoref{def:Seshadri-constants} is compatible with restrictions. We will make use of this in the proof of \autoref{thmB}: a key step in the proof is a lower bound for $\epsint(K_S+4A|_S;x)$ for prime divisors $S\subset X$ in which $x$ is a regular closed point. This is essentially the content of \autoref{propA}.

\section{Lower bounds for Seshadri constants of adjoint bundles on surfaces}\label{sec:surface_case}

In this chapter, we apply the strategy explained in \autoref{subsec:strategy_and_main_results} to the surface case.

\begin{proposition}\label{prop:lower_bound_surfaces}
    Let $S$ be an lci projective surface, let $A,B$ be ample divisors on $S$ and let $x\in S$ be a regular closed point. Then
    \begin{align*}
        \epsint(K_S+3A+B;x)\geq
        \begin{cases}
            \frac{3}{4} &\text{if }B^2=1,\\
            1 &\text{otherwise.}
        \end{cases}
    \end{align*}
    In particular, we always have $\epsint(K_S+4A;x)\geq3/4$.
\end{proposition}

\begin{remark}\label{rem:comp_propA_murayama}
    The hypotheses of \autoref{prop:lower_bound_surfaces} are very similar to the ones of \cite[Theorem 3.1]{Murayama_Effective_Fujita-type_theorems_for_surfaces_in_arbitrary_characteristic}. The subtle difference is that we do not impose $K_S+3A+B$ to be nef. If there are curves in the singular locus of $S$, then \autoref{lem:bend-and-break} does not apply to them, and we couldn't find a way to effectively ensure non--negativity of adjoint divisors along these curves. However, it applies to curves through $x$, as they automatically intersect the regular locus. We will exploit this extra flexibility in the proof of \autoref{thmB}, when working with prime divisors $S$ in a smooth threefold $X$ that are regular at $x$.
\end{remark}

\begin{proof}
    Let $D\in|B|_{\bQ}$ be arbitrary, and let $C\subset S$ be a curve through $x$, with multiplicity $\mu\coloneqq\mult_xC$. We introduce the following notation:
    \begin{align*}
        r\coloneqq\Ard{x}{D},\quad b\coloneqq\Ard{C}{D},\quad b^{\leq 1}\coloneqq\min\{b,1\},\quad\lambda\coloneqq\frac{b^{\leq 1}}{b}=\min\left\{1,\frac{1}{b}\right\}. 
    \end{align*}
    We then have
    \begin{align*}
        K_S+3A+B&\sim_{\bQ} (1-b^{\leq 1})(K_S+3A)+b^{\leq1}(K_S+3A)+(1-\lambda)B+\lambda D\\
        &=(1-b^{\leq 1})(K_S+3A)+(1-\lambda)B+b^{\leq 1}\left(K_S+3A+\frac{1}{b}D\right).
    \end{align*}
    Now, since $x$ is in the regular locus, we have $(K_S+3A).C\geq 0$ by \autoref{rem:K_X+2(dimX)A_nef_surface}. As $B$ is ample and $a^{\leq 1},\lambda\leq 1$, we obtain
    \begin{align*}
        (K_S+3A+B).C&\geq b^{\leq 1}\left(K_S+3A+\frac{1}{b}D\right).C\geq b^{\leq 1}\left(3+\left(K_S+\frac{1}{b}D\right).C\right).
    \end{align*}
    Now notice that $D/b$ is an effective $\bQ$--divisor with $\Ard{C}{D/b}=1$, so by \autoref{lem:adjunction_inequality_surfaces} we have
    \begin{align*}
        (K_S+3A+B).C\geq b^{\leq 1}\left(1+\left(\frac{r}{b}-1\right)\mu\right).
    \end{align*}
    We now further investigate the left hand side: note that
    \begin{align*}
        1+\left(\frac{r}{b}-1\right)\mu&=\mu^2-\mu+1+\frac{1}{b}(r-b\mu)\mu\\
        &\geq \frac{3}{4}\mu^2+\frac{1}{b}(r-b\mu)\mu\\
        &\geq \frac{3}{4}\mu^2+\frac{3}{4}\cdot\frac{1}{b}(r-b\mu)\mu\\
        &=\frac{3r}{4b}\mu,
    \end{align*}
    where in the last inequality, we used that $r\geq b\mu$ by \autoref{lem:order_transitive_restriction}. Finally, multiplying by $b^{\leq 1}$, we obtain
    \begin{align*}
        (K_S+3A+B).C&\geq b^{\leq 1}\frac{3r}{4b}\mu\\
        &=\min\left\{\frac{3r}{4b}\mu,\frac{3r}{4}\mu\right\}\\
        &\geq\min\left\{\frac{3}{4}\mu^2,\frac{3r}{4}\mu\right\}\\
        &=\frac{3}{4}\mu\min\{\mu,r\}
    \end{align*}
    where for the last inequality, we again used $r\geq b\mu$ by \autoref{lem:order_transitive_restriction}. As we chose $D\in|A|_{\bQ}$ arbitrarily, we can take the supremum of $r=\Ard{x}{D}$ as $D$ ranges through the $\bQ$--linear system to obtain
    \begin{align*}
        (K_S+3A+B).C\geq \frac{3}{4}\mu\min\left\{\mu,\sqrt{B^2}\right\}\geq \frac{3}{4}\mu
    \end{align*}
    by \autoref{lem:div_high_mult_irred}. This shows that we always have $\epsint(K_S+3A+B;x)\geq 3/4$.
    
     To conclude, note that, as $(K_S+3A+B).C\geq 1$, we can actually slightly improve the inequality above to 
    \begin{align*}
        (K_S+3A+B).C\geq\max\left\{1,\frac{3}{4}\mu\min\left\{\mu,\sqrt{B^2}\right\}\right\}.
    \end{align*}
    Now note that if $B^2\geq 2$, then we always have
    \begin{align*}
        \max\left\{1,\frac{3}{4}\mu\min\left\{\mu,\sqrt{B^2}\right\}\right\}\geq \mu.
    \end{align*}
    Indeed, this is true for $\mu=1$, and if $\mu\geq 2$ then
    \begin{align*}
        \frac{3}{4}\mu\min\left\{\mu,\sqrt{B^2}\right\}\geq \frac{3}{4}\mu\sqrt{2}\geq\mu.
    \end{align*}
    This shows that $\epsint(K_X+3A+B;x)\geq 1$ if $B^2\geq 2$.
\end{proof}

\begin{remark}\label{rem:surface_case}
    Comparing \autoref{prop:lower_bound_surfaces} with \cite[Theorem 4.1]{Bauer_Szemberg_on_the_Seshadri_constants_of_adjoint_line_bundles} resp.\ \cite[Corollary 3.2]{Murayama_Effective_Fujita-type_theorems_for_surfaces_in_arbitrary_characteristic} might suggest that the bounds in \autoref{prop:lower_bound_surfaces} are an improvement of the constant $1/2$ when applied to adjoint divisors of the form $K_X+4A$. However, following their proof for the divisor $K_S+4A$ actually gives
    \begin{align*}
       \varepsilon(K_S+4A;x)\geq 1
    \end{align*}
    for every point $x$ on a smooth surface $S$. Therefore, if, e.g.,\ $A=B$, $A^2=1$ and $S$ is smooth, then \autoref{prop:lower_bound_surfaces} is not sharp.
\end{remark}

\section{Results on Seshadri constants of adjoint bundles on smooth threefolds}\label{sec:threefold_case}

We now apply the same method as in \autoref{sec:surface_case} to the case where $X$ is a smooth projective threefold. Unfortunately, we are only able to obtain a lower bound after excluding finitely many curves.

\begin{thm}\label{thm:threefold_case}
    Let $X$ be a smooth projective threefold, let $A,B,B'$ be ample divisors on $X$ and let $x\in X$ be a closed point. Then for every $\delta>0$, there exist at most finitely many integral curves $C$ through $x$ such that
    \begin{align*}
        \frac{(K_X+4A+B+B').C}{\mult_x C}\leq \frac{1}{2\sqrt{2}}-\delta.
    \end{align*}
    In particular, if $\varepsilon\coloneqq\varepsilon(K_X+4A+B+B';x)<1/2\sqrt{2}$, then $\varepsilon$ is rational and there exists a Seshadri curve.
\end{thm}

\begin{proof}
    Fix $\delta>0$. Using \autoref{lem:div_high_mult_irred}, choose $D\in|B|_{\bQ}$ such that
    \begin{align*}
        \Ard{x}{D}>\sqrt[3]{B^3}-\delta\geq1-\delta.
    \end{align*}
    Let us write $D=\sum_i b_i S_i$ for prime divisors $S_i\subset X$. For every $i$ such that $\mult_x S_i\in\{2,3\}$, choose $D'_i\in|B'|_{\bQ}$ be such that
    \begin{align*}
        \Ard{x}{D_i'|_{S_i}}> \sqrt{\frac{B'^2.S_i}{\mult_x S_i}}-\delta\geq \frac{1}{\sqrt{\mult_x S_i}}-\delta,
    \end{align*}
    again using \autoref{lem:div_high_mult_irred}. Finally, let $\cC$ be the finite collection of integral curves defined by
    \begin{align*}
        \cC\coloneqq\left\{C\mid \exists i: \left(C\subseteq\Sing(S_i)\right)\lor\left(\mult_xS_i\in\{2,3\}\text{ and }C\subseteq\Supp(D'_i|_{S_i})\right)\right\}.
    \end{align*}
    Let now $C$ be an integral curve through $x$ which is not an element of $\cC$. We are then going to show that
    \begin{align*}
        \frac{(K_X+4A+B+B').C}{\mult_x C}> \frac{1}{2\sqrt{2}}-\delta.
    \end{align*}
    To this end, we introduce the following notation:
    \begin{align*}
        \mu\coloneqq\mult_x C&,\quad \mu_i\coloneqq\mult_x S_i\\
        I\coloneqq\{i\mid C\subset S_i\},&\quad J\coloneqq\{i\mid C\not\subset S_i\}\\
        D_I\coloneqq \sum_{i\in I}b_iS_i,&\quad D_J\coloneqq\sum_{i\in J}b_iS_i\\
        b_I\coloneqq\sum_{i\in I}b_i,\quad b_I^{\leq 1}\coloneqq&\min\{1,b_I\},\quad \lambda\coloneqq \frac{b_I^{\leq 1}}{b_I}.
    \end{align*}
    Note that
    \begin{align*}
        B&\sim_{\bQ} (1-\lambda)B+\lambda D_I+\lambda D_J\\
        &=(1-\lambda)B+b_I^{\leq 1}\sum_{i\in I}\frac{b_i}{b_I}S_i+\lambda D_J
    \end{align*}
    and
    \begin{align*}
        (K_X+4A+B')=(1-b_I^{\leq 1})(K_X+4A+B')+b_I^{\leq 1}(K_X+4A+B').
    \end{align*}
    As $B$ and $K_X+4A+B'$ are ample by \autoref{cor:K_X+2(dimX)A_nef}, we thus obtain
    \begin{align*}
        (K_X+4A+B+B').C&\geq b_I^{\leq 1}(K_X+4A+B').C+b_I^{\leq 1}\sum_{i\in I}\frac{b_i}{b_I}S_i.C+\lambda D_J.C\\
        &=b_{I}^{\leq 1}\sum_{i\in I}\frac{b_i}{b_I}(K_X+S_i+4A+B').C+\lambda D_J.C.
    \end{align*}
    Now, on the one hand, we have by construction $C\not\subset\Supp(D_J)$, and thus
    \begin{align*}
        D_J.C\geq \Ard{x}{D_J}\mu
    \end{align*}
    by \autoref{lem:intersection_lb_multiplicities}. On the other hand, let us investigate the terms $(K_X+S_i+4A+B').C$ by distinguishing different values of $\mu_i$. If $\mu_i\geq 3$, then the bound of \autoref{lem:bend-and-break} will turn out to be good enough. But the components $S_i$ with $\mu_i\leq2$ need special treatment, because \autoref{lem:bend-and-break} gives a non-positive lower bound for the relevant intersection numbers on them, which is not helpful. For components where $\mu_i=1$, we can get by because then $x$ is a regular point on $S_i$ and we can apply \autoref{prop:lower_bound_surfaces}. If $\mu_i=2$, we are a priori in a difficult position, because we can neither apply \autoref{prop:lower_bound_surfaces}, nor \autoref{lem:bend-and-break}. We get around this issue by using that $C\not\subset\Supp(D_i'|_{S_i})$, as $C\notin\cC$. We then get a lower bound simply by using \autoref{lem:intersection_lb_multiplicities}. The reason why we treat components with $\mu_i=3$ similar to those with $\mu_i=2$ is to optimize the numerics, see \autoref{rem:treatment_different_mults}.\medskip
    \begin{itemize}
        \item[$\underline{\mu_i=1}:$] In this case, we have
        \begin{align*}
            (K_X+S_i+4A+B').C=(K_{S_i}+4A|_{S_i}+B'|_{S_i}).C.
        \end{align*}
        By \autoref{prop:lower_bound_surfaces}, we have
        \begin{align*}
            \varepsilon(K_{S_i}+4A|_{S_i}+B'|_{S_i};x)=\varepsilon(K_{S_i}+3A|_{S_i}+(A|_{S_i}+B'|_{S_i});x)\geq 1,
        \end{align*}
        and thus
        \begin{align*}
            (K_X+S_i+4A+B').C\geq \mu=\mu_i\mu.
        \end{align*}\medskip
        \item[$\underline{\mu_i=2,3}:$] In this case, since $C\not\subset\Sing(S_i)$, we have
        \begin{align*}
            (K_X+S_i+4A+B').C\geq (\mu_i-2)\mu+B'.C.
        \end{align*}
        by \autoref{lem:Seshadri_hypersurface_high_mult}. Furthermore, as $C\not\subset\Supp(D_i'|_{S_i})$, we have
        \begin{align*}
            B'.C=D_i'|_{S_i}.C\geq\Ard{x}{D_i'|_{S_i}}\mu\geq \left(\frac{1}{\sqrt{\mu_i}}-\delta\right)\mu
        \end{align*}
        by \autoref{lem:intersection_lb_multiplicities}. Hence,
        \begin{align*}
            (K_X+S_i+4A+B').C&\geq\left(\mu_i-2+\frac{1}{\sqrt{\mu_i}}-\delta\right)\mu\\
            &\geq \left(\frac{1}{2\sqrt{2}}-\frac{\delta}{2}\right)\mu_i\mu,
        \end{align*}
        where the last inequality follows from a simple computation.\medskip
        \item[$\underline{\mu_i\geq 4}:$] In this case, as $C\not\subset\Sing(S_i)$, we have
        \begin{align*}
            (K_X+S_i+4A+B').C&\geq (\mu_i-2)\mu\\
            &\geq\frac{1}{2}\mu_i\mu.
        \end{align*}
    \end{itemize}
    In conclusion, in all cases we have
    \begin{align*}
        (K_X+S_i+4A+B').C\geq\rho\mu_i\mu
    \end{align*}
    where
    \begin{align*}
        \rho\coloneqq \frac{1}{2\sqrt{2}}-\frac{\delta}{2}.
    \end{align*}
    Combining all of the above, we obtain that
    \begin{align*}
        (K_X+4A+B+B').C&\geq b_{I}^{\leq 1}\sum_{i\in I}\frac{b_i}{b_I}\rho\mu_i\mu+\lambda \Ard{x}{D_J}\mu\\
        &=\lambda\rho\mu\Ard{x}{D_I}+\lambda\mu\Ard{x}{D_J}.
    \end{align*}
    Now, if $b_I\geq 1$, then $\lambda=1/b_I$, and thus
    \begin{align*}
        \lambda\Ard{x}{D_I}=\frac{1}{b_I}\sum_{i\in I}b_i\mu_i\geq\min_{i\in I}\mu_i\geq 1.
    \end{align*}
    So in this case, we have
    \begin{align*}
        (K_X+4A+B+B').C\geq\rho\mu.
    \end{align*}
    On the other hand, if $m_I\leq 1$, then $\lambda =1$, and as $\rho\leq 1$, we have
    \begin{align*}
        \rho\Ard{x}{D_I}+\Ard{x}{D_J}&\geq\rho(\Ard{x}{D_I}+\Ard{x}{D_J})\\
        &=\rho \Ard{x}{D}\\
        &\geq \rho(1-\delta).
    \end{align*}
    So in this case we have
    \begin{align*}
        (K_X+4A+B+B').C\geq\rho(1-\delta)\mu.
    \end{align*}
    To summarize, we have shown that
    \begin{align*}
        \frac{(K_X+4A+B+B').C}{\mult_x C}\geq\left(\frac{1}{2\sqrt{2}}-\frac{\delta}{2}\right)(1-\delta)>\frac{1}{2\sqrt{2}}-\delta
    \end{align*}
    whenever $C\notin\cC$. Therefore, if $C$ is an integral curve through $x$ with
    \begin{align*}
        \frac{(K_X+4A+B+B').C}{\mult_x C}\leq \frac{1}{2\sqrt{2}}-\delta,
    \end{align*}
    then we must have $C\in \cC$, and as $\cC$ is finite, we are done.
\end{proof}

\begin{remark}\label{rem:treatment_different_mults}
    For components with $\mult_x S_i=3$, the argument would have gone through as for those with higher multiplicity. However, this would have resulted in the constant $1/3=0.33...$ instead of $1/2\sqrt{2}=0.35...$, so we decided to go for this slight improvement, by treating them similarly to the components with $\mult_x S_i=2$.
\end{remark}

\begin{remark}\label{rem:mult_2_case}
    Note that combining \autoref{lem:bend-and-break} and \autoref{prop:lower_bound_surfaces} yields universal lower bounds for $\epsint(K_S+4A|_S;x)$ whenever $S\subset X$ has an isolated singularity with multiplicity $\mult_xS\neq2$ at $x$. It is natural to ask if there is a similar lower bound for isolated singularities with multiplicity $2$. The key technical difficulty here is the phenomenon in \autoref{rem:counterexample_order_semicont}: when running the proof of \autoref{prop:lower_bound_surfaces}, we have no universal lower bound for $\Ard{x}{D}/\Ard{C}{D}$, and without this the proof breaks down.

    However, we have some partial results in this direction. If we further suppose that
    \begin{align*}
        \wht{\cO}_{S,x}\cong \factor{k[[x,y,z]]}{(z^2-a(x,y))}
    \end{align*}
    and that the initial part of $a(x,y)\in k[[x,y]]$ is not a square, we can exclude the phenomenon in \autoref{rem:counterexample_order_semicont} by proving that 
    \begin{align*}
        \frac{\Ard{x}{D|_S}}{\Ard{C}{D|_S}}\geq\frac{1}{4}
    \end{align*}
    for every $D\in |A|_S|_{\bQ}$, using the techniques in \cite{Bravo_Encinas_Guillan-Rial_On_some_properties_of_the_asymptotic_Samuel_function}. This will be explored in future work.
\end{remark}

\bibliographystyle{amsalpha} 
\bibliography{includeNice}
\end{document}